\documentclass{amsart}
\usepackage{amssymb,amsmath,latexsym,amsthm,xcolor}%or any usual package
\usepackage[english]{babel}

\usepackage{amsmath,amssymb,amsbsy,amsfonts,amsthm,latexsym,
                       amsopn,amstext,amsxtra,euscript,amscd}
\usepackage{hyperref}
\usepackage{array}
\usepackage{ifthen}
\usepackage{url}

\newtheorem{thm}{Theorem}
\newtheorem{cor}[thm]{Corollary}
\newtheorem{lem}[thm]{Lemma}

\newtheorem{exa}[thm]{Example}

\makeatletter
\let\@@pmod\pmod
\DeclareRobustCommand{\pmod}{\@ifstar\@pmods\@@pmod}
\def\@pmods#1{\mkern4mu({\operator@font mod}\mkern 6mu#1)}
\makeatother

\makeatletter %  kronecker symbol macro; auto-adapts to display mode; usage: \kronecker{p}{q} in math mode; inline or as index
\newcommand*\kronecker[2]{%
  \relax\if@display
    \expandafter{(\frac{#1}{#2})}
  \else
    \expandafter{\bigl(\frac{#1}{#2}\bigr)}
  \fi
}
\makeatother
\begin{document}

\title[Higher Reciprocity Laws and Ternary Linear Recurrence Sequences]{Higher Reciprocity Laws 
and Ternary Linear Recurrence Sequences}

\author[Pieter Moree]{Pieter Moree}
\address{Max Planck Institute for Mathematics, Bonn}
\email{moree@mpim-bonn.mpg.de}
%1401 S, U.S. 421, Westville IN 46391 USA}

%
\author[Armand Noubissie]{Armand Noubissie}
\address{Max Planck Institute for Mathematics, Bonn}
\email{noubissie@mpim-bonn.mpg.de or a.noubissie94@gmail.com}  
%

%\subjclass[2010]{11D41, 11D61}
%
%\keywords{modular forms, elliptic curves, Galois representations, Fermat equation, Exponential Diophantine equation.}

%\date{\today}

\begin{abstract}
We describe the set of prime numbers splitting completely in 
the non-abelian splitting field of certain monic irreducible 
polynomials of degree $3.$ As an application 
we establish some divisibility properties of 
the associated ternary recurrence sequence by 
primes $p$,
thus greatly extending recent work of Evink and Helminck and of Faisant.
We also prove some new results on the number of solutions
of the characteristic equation of the recurrence sequence 
modulo $p,$ extending and
simplifying earlier work of Zhi-Hong Sun (2003).
\end{abstract}

\maketitle

%%%%%%%%%%%%%%%%%%%%
\section{Introduction}\label{sec:intro}
%%%%%%%%%%%%%%%%%%%%
The Tribonacci sequence 
$(T_n)_{n\ge 1}$ is a generalization of the Fibonacci 
sequence. It is defined by $T_1=1, T_2=1, T_3=2$, and 
by the third 
order linear recurrence 
$T_n=T_{n-1}+T_{n-2}+T_{n-3}$ for every $n\ge 4.$
Using class field theory, Evink and Helminck \cite{EH} proved the 
following intriguing result.
\begin{thm}
\label{Dutch}
Any prime $p\nmid 11\cdot 19$ divides $T_{p-1}$ if and only if 
it is represented by the binary form 
$X^2+11Y^2.$
\end{thm}
Here and in the rest of the paper the letter $p$ will
be exclusively used to denote prime numbers.
\par More recently Faisant \cite{F} 
established a similar result for the Padovan sequence 
 defined by $B_0=0, B_1=B_2=1$, and 
by 
%the third order linear recurrence sequence
$B_n=B_{n-2}+B_{n-3}$ for every $n\ge 3.$
\begin{thm}
\label{Faisant}
A prime $p$ be divides  $B_{p-1}$ if and only if 
$p$ is represented by the binary form 
$X^2+23Y^2$ with $X\ne 0$ and $Y$ integers.
\end{thm}
Recall that the Ramanujan tau function 
$\tau(n)$ is defined as the coefficient of $q^n$ in the formal series
expansion of 
$q\prod_{k=1}^{\infty}(1-q^k)^{24}.$
The Ramanujan-Wilton \cite{wilton} congruences state that
modulo 23 we have
$$
\tau(p)\equiv
\begin{cases}
\phantom{-}1\ &\textrm{if}\ p=23;\\
\phantom{-}0\ &\textrm{if}\ \kronecker{p}{23} =-1;\\
\phantom{-}2\ &\textrm{if}\ p=X^2+23Y^2, \,X\ne 0;\\
-1\ &\textrm{otherwise}.\\
\end{cases}
$$
Denote by
$N_p(f)$ the number of distinct roots modulo $p$ of
a polynomial $f\in \mathbb Z[x].$ It is known that
$\tau(p)\equiv N_p(x^3-x-1)-1\pmod*{23},$ cf.\,Serre \cite[p.\,437]{serreJordan} 
or \cite[pp.\,42--43]{123}.
This then leads to the following 
reformulation of Faisant's theorem.
\begin{thm}
\label{thm:padua}
A prime $p$ divides  
$B_{p-1}$ if and only if
$N_p(x^3-x-1)=3$ if and only if
$\tau(p)\equiv 2 \pmod*{23}.$
\end{thm}
\noindent The observant reader will notice that $x^3-x-1$ is the characteristic polynomial
of the Padovan recurrence.
\par We point out that also Theorem 
\ref{Dutch} can be reformulated 
using a congruence observed by Ahlgren \cite[(5.4)]{Ahlgren}.
\begin{thm}
Let $r_{12}(n)$ be the number of representations of
$n$ as a sum of twelve squares. 
A prime $p$ divides $T_{p-1}$ if and only if 
$r_{12}(p)\equiv 4 \pmod*{11}.$
\end{thm}
Let $f(x)\in \mathbb Z[x]$ be a monic irreducible polyonmial.
Any rule for characterizing those primes $p$ for which $N_p(f)=\text{deg}(f)$ might be
called a \emph{higher reciprocity law}. In case $f$ is quadratic, the law of quadratic reciprocity leads to a characterization in terms of congruence classes. By class field
a characterization in terms of congruence classes 
exists if and only if the splitting field of $f$ over the
rationals is abelian (see, e.g.\,\cite[Chp.\,1]{HS}). 
An example is given by Theorem \ref{thmH} below.
In the non-abelian case sometimes there is a characterization 
involving Fourier coefficients of a modular form. 
Here our aim is to find
higher reciprocity laws involving linear ternary recurrence sequences. Theorem 
\ref{thm:padua} gives a nice demonstration of what we are after.
\par Results similar to Theorems \ref{Dutch} and \ref{Faisant}, but involving \emph{binary} recurrences 
$\{u_n\}_n$ are known
since Cauchy (1829) (see also Sun \cite{sun2} and
Williams and Hudson \cite{wuhu}). Typically they involve distinguishing between
$u_{(p-1)/3}$ if $p\equiv 1\pmod*{3}$ and $u_{(p+1)/3}$ if $p\equiv 2\pmod*{3}.$
It can also be $u_{(p-1)/4}$ and $u_{(p+1)/3},$ as, e.g., in Halter-Koch \cite{H-K}.
Further, the coefficients of the characteristic polynomial might be large. 
E.g., in \cite{wuhu} the sequence $u_0=2,~u_1=529,$ and $u_{n+2}=529u_{n+1}-40^3u_n$
plays a role. 
\par In this paper we provide more results 
in the same vein as Theorems \ref{Dutch} and \ref{Faisant}. 
All starting
terms and coefficients will be no larger than $5$ in absolute value and no distinction into
two cases is necessary.
We show
that Theorem \ref{Dutch} belongs to a family of
similar results, which we present in
Table \ref{tab:2}.  Likewise, the family for Theorem \ref{Faisant} is presented in Table \ref{tab:1}.
\par Our method of proof is by relating the ternary sequence divisibility by a prime $p,$ to the 
number of solutions of modulo $p$ of the characteristic equation 
(given
Theorem \ref{thm:padua} this becomes perhaps not as a surprise). This quantity on its
turn we relate using class field theory to the representation of primes by binary quadratic
forms. In the next two subsections we present our
results and gives some corollaries in the rest of the introduction.
\subsection{Numbers of solutions of cubic polynomials modulo primes}
\label{sec:npf}
The following result is due to Sun \cite{sun}, 
who proved it by 
an elementary method (taking about fifteen pages). 
Before tackling some relevant variants, 
we provide a short and simple reproof of his result 
using class field theory.
\begin{thm}
\label{thm3"}
	For arbitrary integers $a_1,a_2$ and $a_3,$ let $\{s_n\}$ be the third-order 
	recurrence sequence defined by 
	 $$s_0=3,~s_1=-a_1,~s_2=a_1^2-2a_2,~~s_{n+3}+a_1s_{n+2}+a_2s_{n+1}+a_3s_n=0 \quad (n\geq 0).$$
	Let $f(x)=x^3+a_1x^2+a_2x+a_3$. Given a prime $p,$ we 
	 let $N_p(f)$ denote the number of 
	 integers $0\le a<p$ for which $f(a) \equiv 0 \pmod*{p}$.
	 If $p \nmid \rm 6\,disc(f)(a_1^2-3a_2)$, then
	 $$
	 N_p(f)=
	 \begin{cases}
	 \phantom{-}3\ &\textrm{if}\quad s_{p+1}\equiv a_1^2-2a_2 \pmod*{p};\\
	 \phantom{-}0\ &\textrm{if}\quad  s_{p+1}\equiv a_2 \pmod*{p};\\
	 \phantom{-}1\ &\textrm{otherwise.}%\quad  s_{p+1}\not \equiv a_2,~~a_1^2-2a_2  \pmod*{p}.
	 \end{cases}
	 $$
\end{thm}
Unfortunately this result does not cover the Padovan and Tribonacci sequences. 
For this reason we will establish two analogues 
that do cover these sequences.
These results are by no means exhaustive; 
using our class field theory based approach it is easy to establish
similar results.
\begin{thm}\label{thm1"}
	Let  $a_2$ and $a_3$ be integers and 
	let $\{u_n\}$ be the third-order recurrence sequence defined by 
	\begin{equation}
	\label{eq:u}
	u_0=0,~u_1=-a_2,~u_2=-a_3,~~u_{n+3}+a_2u_{n+1}+a_3u_n=0 \quad (n\geq 0).
	\end{equation}
	Let $f=x^3+a_2x+a_3$ and $d=9a_3^2-4a_2^3$. 
	Put $D={\rm disc}(f).$
	Let 
	\begin{equation}
	   \label{long} 
	p \nmid 6Da_2a_3((20a_2^3a_3+27a_2^3+9a_2d)^2-d(31a_2^2+d)^2),
	\end{equation}
	be a prime. Then 
	$$
	N_p(f)=
	\begin{cases}
	\phantom{-}3\ &\textrm{if}\quad Du_{p-1}^2 \equiv 0 \pmod*{p};\\
	\phantom{-}0\ &\textrm{if}\quad  Du_{p-1}^2\equiv a_2^4 \pmod*{p};\\
	\phantom{-}1\ &\textrm{otherwise.} %\quad  \rm disc(f)u_{p-1}^2\not \equiv 0,~~a_2^4  \pmod*{p}.\\
	\end{cases}
	$$
\end{thm}
Using the same approach as in the proof of Theorem \ref{thm3"}, we deduce the following result

\begin{thm}\label{thm2"}
		Let  $a_1, a_2$ and $a_3$ be integers and 
	let $\{U_n\}$ be the third-order recurrence sequence defined by 
	$$U_0=0,~U_1=1,~U_2=-a_1,~~U_{n+3}+a_1U_{n+2}+a_2U_{n+1}+a_3U_n=0 \quad (n\geq 0).$$
	Let $f=x^3+a_1x^2+a_2x+a_3$. 	Put $D={\rm disc}(f).$ Given
	a prime $p$ we have
	$$
	N_p(f)=
	\begin{cases}
	\phantom{-}3\ &\textrm{if}\quad  DU_{p-1}^2 \equiv 0 \pmod*{p};\\
	\phantom{-}0\ &\textrm{if}\quad   DU_{p-1}^2\equiv (a_1^2-3a_2)^2 \pmod*{p};\\
	\phantom{-}1\ &\textrm{otherwise,}%\quad  \rm DU_{p-1}^2\not \equiv 0,~~(a_1^2-3a_2)^2  \pmod*{p}.\\
	\end{cases}
	$$
	with at most finitely many exceptions.
\end{thm}
Saito \cite{Saito} gave similar results for polynomials $f$ of arbitrary degree, however restricting to
the case where $N_p(f)=\text{deg}(f).$ In the cubic case Sun \cite{sun3} also related $N_p(f)$ to a sum involving the binomial coefficient $\binom{3k}{k}.$
\par To the uninitiated the problem of determining $N_p(f)$ might even seem somewhat 
recreational, however we like to point out the following quote of Dalawat \cite[p.\,32]{Dalawat}: 
``The Langlands programme at its most basic level, is a search of patterns in the sequence
$N_p(f)$ for varying primes $p$ and a fixed but arbitrary polynomial with rational coefficients.".
Our main motivation in this paper is to understand for which
binary forms
$X^2+nY^2$ with $n\ge 1,$ similar results to those
of Evink-Helminck and Faisant exist.
 We will show
that this problem is closely related to the 
question of for which 
$n$ the class number of $\mathbb Q(\sqrt{-n})$ is
$1$ or $3$.  More precisely, we are looking for the monic irreducible polynomials $f$ of degree $3$ with integer coefficients, such that,  with at most finitely many
exceptions, all the primes that split completely over the splitting field of $f$ are represented by the same principal form of discriminant $-D$. This 
is a problem with a 
long tradition and has been studied by many mathematicians. In 1827, Jacobi  showed that for all prime greater than $3$ such that $p \nmid 243$, $N_p(x^3-3)=3$ if and only if $p$ is represented by the principal form $X^2+XY+61Y^2$. Gauss (published in 1876 in his
Collected Works) showed that for 
every
prime $p \nmid 2\cdot 3\cdot 27$ we have $N_p(x^3-2)=3$ if and only if $p$ can be written in the form $p=X^2+27Y^2,$ 
cf.\,Cox \cite[Theorem 4.15]{Cox}.  
In 1868, Kronecker proved that  
$p \nmid 2\cdot 3\cdot 31$ 
can be written in the form $p=X^2+31Y^2$ if and only if
$N_p(x^3+x+1)=3.$ In 1991, Williams and Hudson in \cite{wuhu} found $25$  monic irreducible polynomials of degree $3$ with integer coefficient $f$ such that, with
 finitely many exceptions, the primes $p$ splitting completely over the splitting field of $f$ are represented by the same principal form of discriminant $-D$.
 \par A somewhat related problem was studied by
 Ciolan et al.\,\cite{clum}, who showed that for a large
 class of ternary sequences $\{U_n\}$, including the Tribonacci sequence, one has
 $$\#\{n\le x:U_n=X^2+nY^2\text{~for~some~integers~}X,Y\}
 \ll \frac{x}{(\log x)^{0.05}}.$$
 \par A lot of (historical) material on primes of the form $X^2+nY^2$ can 
 be found in the 
 beautiful book by Cox \cite{Cox}. For more on the characterization of these
 primes using Fourier coefficients of modular forms, see. e.g., the book by Hiramatsu
 and Saito \cite{HS}.
\subsection{Connections with class field theory}
\label{sec:CFT}
In this paper, we provide a large class of monic irreducible polynomials $f$ of degree $3$ with integer coefficients, such that $N_p(f)=3$ if and only if $p$ is represented by the same principal form of discriminant $-D,$ with the exception of finitely many $p$. 
To do so, we give in the 
following theorem a description of 
those values of $n$ for which the splitting field of $f$ is the 
ray class field of $\mathbb{Q}(\sqrt{-D})$ modulo its 
conductor 
(which is either $\langle 1\rangle$ or $\langle 2\rangle,$ due to
the requirement that the odd part of $\text{disc}(f)$ is assumed
to be square-free).
\begin{thm}\label{thm1}
Suppose that $f(x)= x^3+ax^2+bx+c\in \mathbb{Z}[x]$ is
irreducible 
with ${\rm disc}(f)=-4^tn,$ where $t \geq 0,$ and $n$ 
is a positive odd square-free integer. 
Let 
$L$ be the splitting field of $f$ and $\mathfrak{F}$ 
the conductor of $L/K,$ where $K= \mathbb{Q}(\sqrt{-n}).$ 
We denote by $M$ the number field $\mathbb{Q}(\alpha)$ 
with $\alpha$ the unique real root of $f,$
and 
by $d_M$ its discriminant. The conductor $\mathfrak{F}$ is given as follows:
\begin{itemize}
		\item If $2 \mid d_M$ and $-n \equiv 1 \pmod*{8}$, then $\mathfrak{F}= \mathfrak{P}_1'\mathfrak{P}_2'$, where $\mathfrak{P}_1',\mathfrak{P}_2'$ are the prime ideals of $\mathfrak{O}_K$ above $2$. 
		\item If $2 \mid d_M$ and $-n \equiv 5 \pmod*{8}$, then $\mathfrak{F}= \langle 2 \rangle,$ where   $\langle 2 \rangle$ is a prime ideal in $\mathfrak{O}_K$. 
	\item If $\langle 2 \rangle=\mathfrak{P}_1^2 \,\mathfrak{P}_2$ and $-n \equiv 3 \pmod*{4}$, then $\mathfrak{F}=\langle 1 \rangle,$ where  $\mathfrak{P}_1, \mathfrak{P}_2$ are the primes ideals of $\mathfrak{O}_M$. 
		\item If $\langle 2 \rangle=\mathfrak{P}^3 $ and $-n \equiv 3\pmod*{4}$, then $\mathfrak{F}= \mathfrak{P}_1$, where $\mathfrak{P}_1$ is a prime ideal of $\mathfrak{O}_K$ above $2$ and $\mathfrak{P}$ a prime ideal of $\mathfrak{O}_M$.  
		\item If $2 \nmid d_M$, then $\mathfrak{F}= \langle 1 \rangle.$
		
	\end{itemize}
If $h_K=1$ in the second case, or $h_K=3$ in any of the other cases, then
$L$ is the ray class field of $K$  modulo the corresponding conductor $\mathfrak{F}$.
\end{thm}
The positive odd square-free integers $n$ for which $h_K=1,$ respectively $3$ are precisely
\begin{itemize}
\item $1,2,3,7,11,19,67,49,163,$ respectively
\item $23,31,59,83,107,139,211,283,307, 331,379,499,547,643,883,907$ 
\end{itemize}
The first assertion is the celebrated 
Baker-Heegner-Stark theorem, cf.\,Oesterl\'e \cite{oester}, the
second follows using Sage math software and  the work of 
Watkins \cite{Wa}. 
\par We remark that the method of proof of 
Theorem \ref{thm1} also works if
we fix a prime $p$ and 
consider irreducible polynomials $f$ of degree $3$ with
$\text{disc}(f)=-p^{2t}n,$ where $t \geq 0,$ and $n$ 
is a positive  square-free integer coprime to $p.$
\par If $f$ is as in 
Theorem \ref{thm1}, then $N_p(f)=3$  if and only if $p$ splits completely over the splitting field of $f$ (see \cite[Corollary 4.39]{MB}). 
If we are in one of the five cases of
Theorem \ref{thm1} a more explicit description of these completely splitting primes $p$ can be given.
\begin{thm}\label{thm3}
	Assume that we are in one of the five cases
	of Theorem \ref{thm1}. Then the following assertions are equivalent:
	\begin{itemize}
		\item $p$ splits completely over $L$.
		\item $\kronecker{-n}{p}=1$ and $f\equiv 0 \pmod*{p}$ has a solution in $\mathbb{Z}/p\mathbb{Z}$.
		\item $mp=X^2+nY^2,$ with $X,Y \in \mathbb{Z}.$
		\end{itemize}
Here $m=4$ if  $2 \nmid d_M$, $-n \equiv 5 \pmod*{8}$ and $h_K=3,$
 and $m=1$ otherwise.
\end{thm}
 
By combining Theorem \ref{thm3} and Theorem \ref{thm1"}, we obtain a generalization of Theorem \ref{Faisant}. 
Also by combining Theorem \ref{thm2"} and Theorem \ref{thm3}, we obtain a generalization of Theorem \ref{Dutch}.
In terms of ternary recurrences these generalizations have consequences that
are listed in Tables \ref{tab:1} and \ref{tab:2}.

\begin{table}[ht]
  \caption{Results of the form $p\mid u_{p-1} \Leftrightarrow p=(X/2)^2+n(Y/2)^2$, $2\mid X+Y$ with
$u_{n+3}+a_2u_{n+1}+a_3u_{n}=0$, $u_0=0$, $u_1=-a_2$, $u_2=-a_3$.}
\label{tab:1}
\begin{tabular}{|l|l|l|l|}
  \hline
	  $n$ & $(a_1,a_2,a_3)$ &${\rm disc}(f)$& Exceptional primes \\
	\hline \hline
	$23$ & $(0,-1,1)$ & $-23$ & $\{3,23\}$\\
	\hline
	
	$31$ &$(0,1,1)$ & $-31$ & $\{3,31\}$\\
	
	\hline
	$59$ & $(0,2,1)$& $-59$ & $\{2,3,59\}$\\
	
	\hline
	$211$  & $(0,-2,3)$ & $-211$ & $\{2,3,211\}$\\
	\hline
	$283$  & $(0,4,1)$ &$-283$ & $\{2,3,283\}$\\

	\hline
	$499$ & $(0,4,3)$ & $-499$ & $\{2,3,499\}$\\
	
	\hline
	$643$ & $(0,-2,5)$ & $-643$ & $\{2,3,5,643\}$\\
	
	\hline
\end{tabular}
\end{table}

\begin{table}[ht]
\caption{Results of the form $p\mid u_{p-1} \Leftrightarrow 4p=X^2+nY^2$,  with
$u_{n+3}+a_1u_{n+2}+a_2u_{n+1}+a_3u_{n}=0,~~ u_0=0,~ u_1=1, ~u_2=-a_1$.}
\label{tab:2}
\begin{tabular}{|l|l|l|l|}
	\hline
	  $n$ & $(a_1,a_2,a_3)$ &${\rm disc}(f)$& Exceptional primes \\
	\hline \hline
	$83$ & $(1,1,2)$ & $-83$ & $\{2,3,47,83\}$\\
	\hline
	
	$107$ &$(1,3,2)$ & $-107$ & $\{2,3,7,107\}$\\
	
	\hline
	$139$ & $(-1,1,2)$& $-139$ & $\{2,3,47,139\}$\\
	
	\hline
	$307$  & $(-1,3,2)$ & $-307$ & $\{2,3,7,307\}$\\
	\hline
	$331$  & $(-2,4,1)$ &$-331$ & $\{2,3,5,17,331\}$\\

	\hline
	$379$ & $(1,1,4)$ & $-379$ & $\{2,3,101,379\}$\\
	
	\hline
	$547$ & $(1,-3,4)$ & $-547$ & $\{2,3,7,547\}$\\
	
	\hline
	$883$ & $(5,-5,2)$ & $-883$ & $\{2,3,5,421,883\}$\\
	
	\hline
	$907$ & $(5,1,2)$ & $-907$ & $\{2,3,5,11,19,907\}$\\
	\hline
\end{tabular}
\end{table}
\subsection{Applications of Theorem \ref{thm1}} 
Theorem \ref{thm1} can be applied to some well-known sequences to give
results similar to Theorems \ref{Dutch} and \ref{Faisant}.
\begin{cor}\label{cor1}
Consider the Perrin sequence defined by: 
$$P_0=3,~P_1=0,~P_2=2,~~P_{n+3}=P_{n+1}+P_n\text{~for~}n\ge 0.$$ Let $p$ be a prime integer such that 
$p \nmid 2\cdot 3\cdot 23$. The following assertions are equivalent:
\begin{enumerate}
	\item $P_{p+1} \equiv 2 \pmod*{p}$.
	\item $p=X^2+23Y^2$, with  $X,Y \in \mathbb{Z}$.
	\item $\kronecker{-23}{p}=1$ and $x^3-x-1 \equiv 0 \pmod*{p}$ has a solution in $\mathbb{Z}/p\mathbb{Z}$.
\end{enumerate} 
\end{cor}
\begin{cor}\label{cor2}
	Consider the Berstel sequence defined by: 
	$$B_0=B_1=0,~B_2=1,~~B_{n+3}=2B_{n+2} -4B_{n+1}+4B_n\text{~for~}n\ge 0.$$ Let $p$ be a prime integer such that $p \nmid 2\cdot 3\cdot 11  \cdot 13$. The following assertions are equivalent:
	\begin{enumerate}
		\item $B_{p} \equiv 0 \pmod*{p}$.
		\item $p=X^2+11Y^2$, with  $X,Y \in \mathbb{Z}$.
	\end{enumerate} 
\end{cor}
The Berstel sequence has six zero-terms, which is the maximum
number of zero-terms for a non-zero ternary linear recurrence 
by a result of Beukers \cite{B}. 
\par Let $\tau_{16}(n)$ be the coefficient of $x^n$ in the formal series
expansion of 
$$x\big(1+240\sum_{k=1}^{\infty}x^k\sum_{d\mid k}d^3\big)\prod_{k=1}^{\infty}(1-x^k)^{24}.$$
(In modular forms parlance 
$\tau_{16}(n)$ is the $n$-th Fourier coefficient of
the cusp form $\Delta E_4$ of weight 16.)
 \begin{cor}\label{cor5}
 	Consider the ternary recurrence  sequence given by: $C_0=0,~C_1=0,~C_2=1,~~C_{n+3}=C_{n+2} +C_n$. 
 	Let $p \nmid 2\cdot 3\cdot 29\cdot 31$ be a prime. The following assertions are equivalent:
 	\begin{enumerate}
 		\item $C_{p} \equiv 0 \pmod*{p}$.
 		\item $p=X^2+31Y^2$, with  $X,Y \in \mathbb{Z}$.
 		\item $\tau_{16}(p)\equiv 2\pmod*{31}.$
 	\end{enumerate} 
 \end{cor}\label{cor6'}

 \subsection{Abelian splitting fields} For a large class of abelian extensions of degree $3$, Huard and al.\,\cite{H} determined explicitly the set of 
rational primes $p$ splitting completely in
them, as well as the exceptional  primes. 
Here we recall their result. 
\begin{thm}\label{thmH}
Let $f$ be a monic irreducible polynomial of degree $3$ given by $f:=x^3+Ax+B,$ where the largest positive $k$ such that $k^2 \mid A$ and $k^3 \mid B$ is $1$. Assume that the splitting field $L$ of $f$ is abelian (this is so if and only if ${\rm disc}(f)$ is a square). Then, if $p\nmid 3\,{\rm disc}(f)$, 
$$p\text{~splits~completely~in~}L 
\iff p \equiv a_1,a_2, 
\ldots, a_{\varphi(F)/3}\pmod*{F},$$
where $\varphi$ is 
Euler's totient function and the $F,a_i's$ are given in \cite[pp. 468--469]{H}.
\end{thm}

Combining this result and Theorem \ref{thm1"}, we get the following corollary

\begin{cor}\label{cor2'}
   Let $f$ be a monic irreducible polynomial of degree $3$ given by $f:=x^3+Ax+B,$ and $D=\rm disc(f)$, where the largest positive $k$ such that $k^2 \mid A$ and $k^3 \mid B$ is $1$. Let $(N_n)_n$ be a ternary recurrence sequence given by: 
   $$N_{n+3} = -AN_{n+1}-BN_n,\quad \mbox{with}\quad  N_0=0, ~N_1=-A, ~N_2=-B.$$ 
   %With finitely many exceptions of prime integer $p$, we have, modulo $p$
   Modulo a prime $p,$ we have, with finitely many exceptions,
   $$
	DN_{p-1}^2\equiv 
	\begin{cases}
	\phantom{-}0\ &\textrm{if}\quad  p \equiv a_1,a_2, 
\ldots, a_{\varphi(F)/3}\pmod*{F};\\
	\phantom{-}a_2^4\ &\textrm{otherwise.}
	%\quad  \rm DU_{p-1}^2\not \equiv 0,~~(a_1^2-3a_2)^2  \pmod*{p}.\\
	\end{cases}
	$$
   
  \end{cor} 
  
\begin{exa} {\rm  
Let $f$ be the irreducible polynomial $x^3-31x+62$
having discriminant $4^2 \cdot 31^2$.  We define $(N_n)_n$ as follows:
$$N_0=0,~~N_1=31, ~~ N_2=-62 \quad \mbox{and} \quad N_{n+3} = 31N_{n+1}-62N_n.$$ 
By Corollary \ref{cor2'} we deduce that 
   $$
	4N_{p-1}\equiv 
	\begin{cases}
	\phantom{-}0\,\,\,\pmod*{p}\ &\textrm{if}\,\, p \equiv 1,2,4,8,15,16,23,27,29,30 \pmod*{31};\\
	\pm 31\pmod*{p}\ &\textrm{otherwise,}
\end{cases}
$$ 
with finitely many exceptions.}
\end{exa}

%#################################%
\section{preliminaries}\label{sec2}
%#################################%

Let $L$ be a finite Galois extension of 
a number field $K$  and $\mathfrak{P}$ a prime ideal of $\mathcal{O}_K$. 
It is well known (see for example Theorem 3.34 in \cite{Mi}) that there are positive integers $e, f, g$ such that 
$$\mathfrak{P}\mathcal{O}_L=(\mathfrak{B}_1\mathfrak{B}_2\cdots \mathfrak{B}_g)^e,\quad \left[  \mathcal{O}_L/\mathfrak{B}_i:\mathcal{O}_K/\mathfrak{P} \right]  = f,$$ 
for $i=1,\ldots,g \; \mbox{and} \; efg = [L:K],$ 
where the $\mathfrak{B}_i's$ are maximal ideals of $\mathcal{O}_L$. 
The integer $e$ is called the \emph{ramification degree} 
and
is the number  of times the 
maximal ideal $\mathfrak{B}_i$ of $\mathcal{O}_L$ that lies over $\mathfrak{P}$ repeats as a factor of  $\mathfrak{P}\mathcal{O}_L.$
In case $e>1$ one says 
that $\mathfrak{P}$ \emph{ramifies} in $K$.
The integer $g$ is called the 
\emph{decomposition index} and is the number of distinct prime ideals $\mathfrak{B}_i$ over $\mathfrak{P}$. Now we recall the Dedekind-Kummer theorem (see, for example, 
\cite[Theorem 10.1.5]{A}).
\begin{thm}
Let $K$ be an algebraic number field. Then a rational  prime ramifies in $K$ if and only 
if it divides the discriminant of $K$ over the rationals.
\end{thm}
It tells us that there are only finitely 
many ramified $\mathfrak{P}.$ 
Given a maximal ideal $\mathfrak{B}$ 
of $\mathcal{O}_L$ over $\mathfrak{P}$, the \emph{decomposition group} of  $\mathfrak{B}$ is defined as the subgroup of the Galois group that fixes  $\mathfrak{B}$ as a set, that is 
	$$D_{\mathfrak{B}}=\{ \delta \in {\rm Gal}(L/ K)~:~ \mathfrak{B}^{\delta}=\mathfrak{B}\}.$$ 

It is well known that the quotient space Gal$(L/K)/D_{\mathfrak{B}}$ has cardinality $g$, and hence the decomposition group has order $ef$. Given $ \delta \in {\rm Gal}(L/K)$, we define $\overline{\delta}$ by
$$\overline{\delta}(x+\mathfrak{B})= \delta (x) + \mathfrak{B}\quad \mbox{for~all}~~ x \in \mathcal{O}_K.$$
 The map $$\psi\,:\, D_{\mathfrak{B}} \rightarrow \text{Gal}\left( \mathcal{O}_L/\mathfrak{B} ~/~\mathcal{O}_K/\mathfrak{P} \right)$$  defined by $$\psi (\delta)=\overline{\delta},$$ is a surjective homomorphism (see for example  Proposition 9.4 in \cite{JN}) and its 
 kernel $I_{\mathfrak{B}}$ is given by
$$I_{\mathfrak{B}}= \{ \delta \in D_{\mathfrak{B}}\,:\,\delta (x) \equiv x \pmod*{\mathfrak{B}}  \quad \mbox{for~all}~~ x \in \mathcal{O}_L \}.$$  
It is a subgroup of $D_{\mathfrak{B}}$ called 
\emph{inertia group} and has cardinality $e.$
Any representative of Frobenius is called a 
\emph{Frobenius element} of Gal$(L/K)$ and denoted by $\text{Frob}_{\mathfrak{B}}$. When $\mathfrak{B}$ is unramified, $I_{\mathfrak{B}}$ is trivial as its order is $e$, and in this case,  $\text{Frob}_{\mathfrak{B}}$ is unique. In other words, $\text{Frob}_{\mathfrak{B}}$  is the unique element of $D_{\mathfrak{B}}$ which satisfies 
$$\delta (x) \equiv x^{N(\mathfrak{P})} \pmod*{\mathfrak{B}}, ~~\mbox{for~all}~~ x \in \mathcal{O}_L.$$
If  $\mathfrak{B}, \mathfrak{B}'$ are two maximal ideals lying over $\mathfrak{P}$, then  there exists $\delta \in \text{Gal}(L/K)$ such that  $\mathfrak{B}^{\delta} = \mathfrak{B}'.$ Hence the relation between the corresponding Frobenius element is given by 
$$\text{Frob}_{\mathfrak{B}^{\delta}}= \delta^{-1}~\text{Frob}_{\mathfrak{B}}~\delta.$$ 
One defines $\text{Frob}_{\mathfrak{P}}=[\text{Frob}_{\mathfrak{B}}]$, where $[*]$ is the conjugacy class in $\text{Gal}(L/K).$ When $L/K$ is  
an abelian  extension, this class consists of just one element and we identify $\text{Frob}_{\mathfrak{P}}$ with  $\text{Frob}_{\mathfrak{B}}$, thus $\text{Frob}_{\mathfrak{P}}$ does not depend on the maximal ideal lying above $\text{Frob}_{\mathfrak{P}}.$ 
Recall that $\mathfrak{P}$ is said to
be \emph{wildly ramified} in $L$ if $\gcd (e,p) = 1$, where $p$ is the only prime integer divisible by $\mathfrak{P}$ and $e$ is the ramification degree of $\mathfrak{P}$.
Let $\Im$ be a cycle in $K$ (a formal product $\Im = \Im_0\Im_{\infty}$ of an integral ideal $\Im_0$ with a product $\Im_{\infty}$ of 
some or all of the real primes). The cycle $\Im$ is called the \emph{conductor} of $L/K$ if the following conditions are satisfied:
\begin{itemize}
	\item The only prime ideals of $K$ dividing $\Im$ are those which are ramified 
	in $L$;
	\item For any prime ideal $\mathfrak{P}$ of $K$ dividing $\Im_0$ we have $\nu_{\mathfrak{P}}(\Im_0) \leq 2$;
	\item $\mathfrak{P} \parallel \Im_0$ if and only if $\mathfrak{P}$ is wildly ramified in $L.$ 
\end{itemize}
Denote the conductor of $L/K$ by $\Im$  and
define $I^{S(\Im)}$ to be the group of fractional ideals generated by 
the prime ideals of $K$ not  dividing $\Im_0$. Let $Cl_{\Im}$ be the quotient of $I^{S(\Im)}$  by the subgroup of principal ideals lying in $I^{S(\Im)}$ 
that are generated by an element $a$ such that $a>0$ 
for every real 
prime dividing $\Im$ and $\nu_{\mathfrak{P}}(a-1)\geq \nu_{\mathfrak{P}}(\Im)$ for all integral prime ideals $\mathfrak{P}$ dividing $\Im$. By the \textbf{Artin reciprocity law,} the map $$\rho ~:~ Cl_{\Im} \rightarrow \text{Gal}(L/K)$$  defined by $\rho ([\mathfrak{P}])= \text{Frob}_{\mathfrak{P}},$ is surjective. Moreover, there exists an abelian extension $H_{\Im}$ called 
\emph{Hilbert class field} modulo $\Im$ such that $Cl_{\Im} $ is isomorphic to Gal$(H_{\Im}/K)$ and a prime ideal $\mathfrak{P}'$ of $K$ splits completely in $H_{\Im}$ if and only if $\mathfrak{P}'$ is generated by an element $a$ such that $a>0$ 
for every real prime dividing $\Im$ and $\nu_{\mathfrak{P}}(a-1)\geq \nu_{\mathfrak{P}}(\Im)$ for all integral prime ideals $\mathfrak{P}$ dividing $\Im$. 
\subsection{More on primes that split completely}

Let $f$ be a monic irreducible polynomial with integer coefficients and $L$ the splitting field of $f$. We are interested in determining "explicitly" the primes $p$ which split completely over $L$.  
When $L$ is abelian it
is well-known that, with 
possibly finitely many exceptions, $p$ splits completely over $L$ if and only if the Frobenius element to $p$ (denoted $\sigma_p$) is trivial (cf.\,Wyman \cite{W}).
We give two examples illustrating this result.
\begin{exa}[Quadratic fields] {\rm Here $L= \mathbb{Q}(\sqrt{m}),$ 
with $m$ a square-free integer. Let $p \nmid 2m$ be
a prime and $\mathfrak{P}$ a prime ideal of 
$L$ above $p$ of norm $N\mathfrak{P}$. The Frobenius element is the unique element $\sigma_p$ of 
the decomposition group $D(\mathfrak{P})$ such that  for any $\alpha \in \mathcal{O}_L,$  $$\sigma_p(\alpha) \equiv \alpha^{N\mathfrak{P}} \pmod*{\mathfrak{P}}.$$ 
It follows by Euler's criterion that $\sigma_p = {\rm id}_L$ if and only if 
$\kronecker{m}{p}=1.$
For the theory of binary recurrence sequences 
this implies for example that if 
$p\nmid 2\cdot 5$ is a prime, then $F_{p-1} \equiv 0 \pmod*{p}$ if and only if 
$\kronecker{5}{p}=1$, which by the law of quadratic reciprocity is 
equivalent with $p \equiv \pm 1\pmod*{5}.$ A similar, more general, result 
is (see for example \cite{R}, pages 12): If $p$ is a prime integer and $(U_n)_n$ a binary recurrence sequence given by
$$U_{n+2} = PU_{n+1}-Q U_{n}, \quad U_0=0, \quad U_1=1, \quad n\geq 0,$$ then 
$$U_{p-\kronecker{D}{p}}\equiv 0 \pmod*{p} \quad \mbox{if} \quad  p \nmid PQD, \quad \mbox{where} \quad  D=P^2-4Q.$$}
\end{exa}
\begin{exa}[Cyclotomic fields] {\rm Here $L= \mathbb{Q}(\zeta_{m}),$ 
for some odd integer $m>1.$ We know that Gal$(L/\mathbb{Q}) \simeq \left( \mathbb{Z}/m\mathbb{Z}\right)^*$ with $[n]$ acting as $\zeta \mapsto \zeta^n$. 
Any prime 
$p\nmid m$ is unramified in $L$. Hence the Frobenius element 
in $p$ is trivial if and only 
if $[p]=[1],$ that is $p \equiv 1 \pmod*{m}$. 
We conclude that $p$ splits completely over $L$ if and only if $p \equiv 1 \pmod*{m}$.}
\end{exa}

For an arbitrary abelian extions no 
description of the Frobenius element in $p$ is known.
However, using class field theory, it can be shown that, with finitely many exceptions, the primes which split completely over $L$ lie in certain congruence classes modulo the conductor. In the 
generic case though no explicit description of these congruence classes is known.

%###########################################%
\section{Proof of Theorem \ref{thm3"}}\label{sec:proof1}
%###########################################%
Our proof of Theorem \ref{thm3"} will involve the following lemma.
\begin{lem}\label{lem1"}
	Let $f$ the monic irreducible polynomial with integer coefficients of degree $3$ and 
	$p \nmid \rm 6\,disc(f)$ be a prime.
	Suppose that $\mathfrak{B}$ is a prime ideal 
	above $p$ in the splitting field $L$ of $f$. If $L$ is non-abelian, then 
	$$
	N_p(f)=
	\begin{cases}
	\phantom{-}3\ &\textrm{if}\quad       \# D(\mathfrak{B})=1;\\
	\phantom{-}0\ &\textrm{if}\quad   \# D(\mathfrak{B})=3;\\
	\phantom{-}1\ &\textrm{if}\quad  \# D(\mathfrak{B})=2.\\
	\end{cases}
	$$
	If $L$ is abelian, then
	$$
	N_p(f)=
	\begin{cases}
	\phantom{-}3\ &\textrm{if}\quad       \# D(\mathfrak{B})=1;\\
	\phantom{-}0\ &\textrm{if}\quad   \# D(\mathfrak{B})=3;\\
	\end{cases}
	$$
\end{lem}

\begin{proof}
We only deal with  the case where $L$ is non-abelian, the abelian case being similar and left to
the reader.
The assumption $p \nmid 6\,\text{disc}(f)$ implies 
that $p$ is unramified in $\mathbb{Q}(\alpha)$, where $\alpha$ is a root of $f$. 
By  \cite[Corollary 4.39]{MB} we infer that $p$ 
is unramified over $L$ and so $D(\mathfrak{B})$ is cyclic group. By Lagrange's theorem, we get $\# D(\mathfrak{B}) \in \{1,2,3\}$. If $N_p(f)=3$, then $p$ splits completely over $L$. Hence we deduce that the number of prime ideals of $L$ above $p$ is 
maximal, i.e.\,six, and so $\# D(\mathfrak{B})=1$. If $\# D(\mathfrak{B})=1,$ then $p$ splits completely over $L$ and so over $\mathbb{Q}(\alpha),$ and hence $N_p(f)=3$. 
This shows that $N_p(f)=3$ 
implies that $\# D(\mathfrak{B})=1$. If $\# D(\mathfrak{B})=2$, then 
the 
decomposition index $g=3$ and combining this argument with the fact that $p$ is prime ideal over $\mathbb{Q}(\alpha)$ and $[L:\mathbb{Q}(\alpha)]=2$, yields a contradiction, and so 
$N_p(f)=0$ implies  $\# D(\mathfrak{B})=3$. Now assume $\# D(\mathfrak{B}) =3$. If $N_p(f)=1$, then according to \cite[Proposition 10.5.1]{AW} we have $ \langle p \rangle = \mathfrak{P}_1\mathfrak{P}_2,$ with 
$\mathfrak{P}_1$ and $\mathfrak{P}_2$ prime ideals over $\mathbb{Q}(\alpha)$ 
having inertia degrees $2,$ respectively $1$. Since $[L:\mathbb{Q}(\alpha)]=2,$ 
$g=2$ and the index is a multiplicative function, it follows that $\mathfrak{P}_1$ and $\mathfrak{P}_2$ 
are prime ideals over $L$, which contradicts the fact that $L:\mathbb Q$ is Galois.
\end{proof}
\begin{proof}[Proof of Theorem \ref{thm3"}]
By Binet's formula, we have $s_n=\alpha^n+ \beta^n+ \gamma^n,$ where $\alpha,~\beta,~\gamma$ are the roots of $f$. Assume that $p \nmid \rm 6\,disc(f)(a_1^2-3a_2)$ is a prime, and $\mathfrak{B}$ a prime ideal of $\mathcal{O}_L$ above $p$. By Lemma \ref{lem1"} it suffices to show that 
 $$s_{p+1} \equiv a_1^2-2a_2 \pmod*{p} \Leftrightarrow \# D(\mathfrak{B})=1 \text{\,\,and\,\,} s_{p+1} \equiv a_2 \pmod*{p} \Leftrightarrow \# D(\mathfrak{B})=3.$$
If $\# D(\mathfrak{B})=1,$ then $x^p \equiv x \pmod*{\mathfrak{B}}$ for all $x \in\mathcal{O}_L$ and $\mathfrak{B}$ prime ideal above $L$. Hence modulo $\mathfrak{B}$ we have 
$$
s_{p+1} \equiv \alpha^{p+1}+  \beta^{p+1}+\gamma^{p+1} 
\equiv  \alpha^{2}+  \beta^{2}+\gamma^{2}\equiv a_1^2-2a_2
,$$
where we used the fact that $\alpha+\beta+\gamma=-a_1$ and $\alpha \beta +\alpha \gamma +\gamma \beta=a_2$. If $\# D(\mathfrak{B})=3,$ by symmetry,  we can assume without loss of generality that the Frobenius element is given by $$\sigma(\alpha)= \beta,~~\sigma(\beta)= \gamma,~~\sigma(\gamma)= \alpha.$$ We then get $$s_{p+1} \equiv \alpha \beta +\alpha \gamma +\gamma \beta \equiv a_2 \pmod*p.$$ Assume now that $s_{p+1} \equiv a_1^2-2a_2 \pmod*p$. We want to show that $\# D(\mathfrak{B})=1$. Since $p \nmid (a_1^2-3a_2)$, it follows that $\# D(\mathfrak{B})\neq 3$. If $\# D(\mathfrak{B})= 2$,  by symmetry, we can assume without loss of generality that the Frobenius element is given by $$\sigma(\alpha)= \beta,~~\sigma(\beta)= \alpha,~~\sigma(\gamma)= \gamma.$$ We let $y_1=3\alpha+a_1,y_2=3\beta+a_1$ 
and $y_3=3\gamma+a_1,$ where $y_1,y_2,y_3$ are the roots of the polynomial $x^3-3(a_1^2-3a_2)x-b$, with $b=-2a_1^3+9a_1a_2-27a_3$. Hence 
$$s_{p+1}= \dfrac{1}{3^{p+1}}\left((y_1-a_1)^{p+1}+ (y_2-a_1)^{p+1}+(y_3-a_1)^{p+1}\right).$$ 
Noting that  $(y_1-a_1)^{p+1} \equiv (y_2-a_1)(y_1-a_1) \pmod*{\mathfrak{B}},$ 
$y_1+y_2+y_3=0$ and 
using Fermat's little theorem, we get   
\begin{eqnarray*}
	9(a_1^2-2a_2) &\equiv & (y_2-a_1)(y_1-a_1)+(y_2-a_1)(y_1-a_1)+(y_3-a_1)^2\pmod *{\mathfrak{B}}\\ 
	&\equiv & 2y_1y_2+3a_1^2+y_3^2 \pmod*{\mathfrak{B}}.
\end{eqnarray*}
Hence modulo
$\mathfrak{B}$ we have
$$6(a_1^2-3a_2) \equiv 2y_1y_2+y_3^2
\equiv -2(3(a_1^2-3a_2)+y_3(y_1+y_2))+y_3^2,
$$
which implies $4(a_1^2-3a_2) \equiv y_3^2 \pmod*{\mathfrak{B}}$. By multiplying both sides by $y_3$ and using the fact that $y_3^3-3(a_1^2-3a_2)y_3-b=0$, we get 
$y_3 \equiv b/ (a_1^2-3a_2) \pmod*{\mathfrak{B}}$. So we deduce that $b^2-4(a_1^2-3a_2)^3 \equiv 0 \pmod*{\mathfrak{B}}$, which gives the contradiction as $b^2-4(a_1^2-3a_2)^3=-27\text{disc}(f)$. In conclusion we have  $s_{p+1} \equiv a_1^2-2a_2 \pmod*p$ 
if and only if $\# D(\mathfrak{B})=1$. Now assume that $s_{p+1} \equiv a_2 \pmod*p$. Since $p \nmid (a_1^2-3a_2)$, it suffices to show that $\# D(\mathfrak{B})\neq 2$. Let us assume that $\# D(\mathfrak{B})= 2$, by symmetry, we can assume without loss of generality that the Frobenius element is given by $$\sigma(\alpha)= \beta,~~\sigma(\beta)= \alpha,~~\sigma(\gamma)= \gamma.$$ We get, modulo $\mathfrak{B},$ 
\begin{eqnarray*}
	9a_2 &\equiv & (y_2-a_1)(y_1-a_1)+(y_2-a_1)(y_1-a_1)+(y_3-a_1)^2\\ 
	&\equiv & 2y_1y_2+3a_1^2+y_3^2,
\end{eqnarray*}
 which implies $-3(a_1^2-3a_2) \equiv 
 2y_1y_2+y_3^2 \pmod*{\mathfrak{B}}$. Since $y_1+y_2+y_3=0$ and $y_1y_2+y_3y_2+y_1y_3=-3(a_1^2-3a_2)$, one gets $a_1^2-3a_2 \equiv y_3^2 \pmod*{\mathfrak{B}}$. Multiplying both sides
 of this congruence by $y_3$ and since $y_3^3-3(a_1^2-3a_2)y_3-b=0$, we get $y_3 \equiv -b/ 2(a_1^2-3a_2) \pmod*{\mathfrak{B}}$. Combining this last equivalence and $a_1^2-3a_2 \equiv y_3^2 \pmod {\mathfrak{B}}$, we deduce that $b^2-4(a_1^2-3a_2)^3 \equiv 0 
\pmod*{\mathfrak{B}}$, which leads to a contradiction as $b^2-4(a_1^2-3a_2)^3=-27\,\text{disc}(f),$
and by assumption $p\nmid 6\,\text{disc}(f).$
\end{proof}

%###########################################%
\section{Proof of Theorem \ref{thm1"}}\label{sec:pfthm1"}
%###########################################%
\begin{proof}
We first determine the Binet formula of 
the sequence $\{u_n\}$. It is well known that 
there 
exist $a_1,b_1,c_1 \in \mathbb{Q}(\alpha, \beta, \gamma),$
such that
$u_n=a_1\alpha^n+ b_1\beta^n+ c_1\gamma^n.$ 
This gives rise to the system of equations 
\begin{eqnarray*}
	a_1+b_1+c_1 & = & u_0;\\
	a_1\alpha + b_1\beta +c_1\gamma & = & u_1;\\
	a_1\alpha^2 + b_1\beta ^2+c_1\gamma^2 & = & u_2,
\end{eqnarray*}
which by Cram\'er's rule yields that $a_1,b_1,c_1$ are of the form $\Delta_i/\Delta$, where 
$$
\Delta=\left| \begin{matrix} 1 & 1 & 1 \\
\alpha & \beta & \gamma\\
\alpha^2 & \beta^2 & \gamma^2
\end{matrix}
\right| = (\alpha-\beta)(\alpha-\gamma)(\gamma-\beta),
$$
and the
$\Delta_i$ are the $3\times 3$ determinants obtained 
by replacing the $i$th column in $\Delta$ by $(u_0,u_1,u_2)^T,$ leading to  
$$
\Delta_1= u_0(\beta\gamma^2-\gamma\beta^2)-(u_1\gamma^2-u_2\gamma) + \beta^2 u_1- u_2 \beta,
$$
$$
\Delta_2=(u_1\gamma^2-u_2\gamma)-u_0(\alpha \gamma^2-\alpha^2 \gamma) + (u_2\alpha-\alpha^2u_1),
$$ 
$$
\Delta_3= (u_2\beta-u_1\beta^2)+u_0(\alpha \beta^2-\alpha^2 \beta) - (u_2\alpha-\alpha^2u_1).
$$
Hence we get 
$$\Delta\,u_n = (\gamma-\beta) \alpha^{n+3} + (\alpha-\gamma) \beta^{n+3}+ (\beta-\alpha)\gamma^{n+3}.$$ 
Let $p$ be a prime satisfying \eqref{long} and $\mathfrak{B}\in \mathcal{O}_L$ a prime ideal above $p$. 
Since $\Delta ^2=\text{disc}(f)$, by Lemma \ref{lem1"} it suffices to show that 
$\text{disc}(f)u_{p-1}^2 \equiv 0 \pmod*p$ 
if and only if $\# D(\mathfrak{B})=1$ and $\text{disc}(f)u_{p-1}^2 \equiv a_2^4 \pmod*p$ if and only if $\# D(\mathfrak{B})=3.$  
In case $\# D(\mathfrak{B})=1,$ modulo $\mathfrak{B}$ we have $x^p \equiv x$ for all 
$x \in\mathcal{O}_L$ and so
\begin{eqnarray*}
	\Delta\,u_{p-1} &\equiv &  (\gamma-\beta)\alpha^{3}+  (\alpha-\gamma)\beta^{3}+(\beta-\alpha) \gamma^{3} \\ 
	&\equiv &  (\gamma-\beta)(-a_2\alpha-a_3)+  (\alpha-\gamma)(-a_2\beta-a_3)+(\beta-\alpha) (-a_2\gamma-a_3) \\
	&\equiv & 0 ,
\end{eqnarray*}
where we used the fact that $\alpha, \beta, \gamma$ are the solutions of the equation $x^3+a_2x+a_3=0$. If $\# D(\mathfrak{B})=3,$ by symmetry,  we can assume without loss of generality that the Frobenius element is given by $$\sigma(\alpha)= \beta,~~\sigma(\beta)= \gamma,~~\sigma(\gamma)= \alpha.$$ We then get, modulo $\mathfrak{B}$, 
$$
	\Delta\,u_{p-1} \equiv   (\gamma-\beta)\alpha^{2}\beta +  (\alpha-\gamma)\beta^{2}\gamma+(\beta-\alpha) \gamma^{2}\alpha \equiv a_2^2.
$$
Assume now that $u_{p-1} \equiv 0 \pmod*p$. We want to show that $\# D(\mathfrak{B})=1$. Since $p \nmid a_2$, it follows that $\# D(\mathfrak{B})\neq 3$. If $\# D(\mathfrak{B})= 2$,  by symmetry, we can assume without loss of generality that the Frobenius element is given by $$\sigma(\alpha)= \beta,~~\sigma(\beta)= \alpha,~~\sigma(\gamma)= \gamma.$$ 
Using $\alpha+\beta+\gamma=0$ we get  
 modulo $\mathfrak{B}$,
\begin{eqnarray*}
	0 &\equiv &  (\gamma-\beta) \alpha^{p+2} + (\alpha-\gamma) \beta^{p+2}+ (\beta-\alpha)\gamma^{p+2}\\ 
	&\equiv &  (\gamma-\beta)\beta \alpha^2 + (\alpha-\gamma) \beta^2 \alpha +(\beta-\alpha) \gamma^3 \\
	&\equiv & a_2\gamma (\beta-\alpha).
\end{eqnarray*}
Hence $\gamma \equiv 0 \pmod*{\mathfrak{B}},$ which implies $p \mid a_3$, contradiction. In conclusion, we get 
$u_{p-1} \equiv 0 \pmod*p$ if and only if $N_p(f)=3$ for all prime $p\nmid a_2a_3\rm disc(f)$. Now assume that $\rm disc(f)\,u_{p-1}^2 \equiv a_2^4 \pmod*p$. 
Since $p \nmid a_2$
it is clear that  $\# D(\mathfrak{B})\neq 1.$ If $\# D(\mathfrak{B})= 2$, then $a_2^2 \equiv \gamma^2 (\alpha-\beta)^2$ and so $a_2^2 \equiv -\gamma^2(3\gamma^2+4a_2)$. Using the fact that $\gamma^3+a_2\gamma+a_3=0$, we deduce that $a_2\gamma^2-3a_3\gamma + a_2^2 \equiv 0 \pmod*{\mathfrak{B}}$. Then, there exists $v \in \mathcal{O}_L$ such that 
modulo $\mathfrak{B}$
we have
$$v^2 \equiv 9a_3^2-4a_2^3,~
\gamma \equiv \dfrac{3a_3\pm v}{2a_2} \text{~and~}\alpha \beta \equiv -  \dfrac{2a_2a_3}{3a_3\pm v}.$$ Using the fact that $\alpha \beta + \alpha \gamma + \gamma \beta = a_2,$ we obtain $$ -a_2-  \dfrac{2a_2a_3}{3a_3\pm v} \equiv \gamma^2 \pmod*{\mathfrak{B}}.$$ It follows that  
$(20a_2^3a_3+27a_2^3+9a_2d)^2 \equiv d(31a_2^2+d)^2 \pmod* {\mathfrak{B}},$ 
contradicting assumption \eqref{long}.
\end{proof}

%###########################################%
\section{Proof of Theorem \ref{thm1}}\label{sec:thm1}
%###########################################%
\begin{proof}
Recall that by assumption we 
only consider monic irreducible polynomials $f$ of the form $x^3+ax^2+bx+c 
\in \mathbb{Z}[x]$, 
for which ${\rm disc}(f)=-4^tn,$ with $t \geq 0$ and $n$ square-free 
and odd. Let $\alpha$ the unique real root of $f$ and
$\beta, \gamma$ its complex roots. 
It well known that ${\rm disc}(f)=(\alpha-\beta)^2(\alpha-\gamma)^2(\gamma-\beta)^2$.
Let
$L$ be the splitting field of $f$,  $K= \mathbb{Q}(\sqrt{-n}),$ $M =\mathbb{Q}(\alpha)$ and 
let $d_M={\rm disc}(M/\mathbb{Q}).$ By Galois theory, $[L:\mathbb{Q}]\leq 3!=6$.  Since ${\rm disc}(f)<0$, we deduce that $L= \mathbb{Q}(\alpha, \sqrt{{\rm disc}(f)})= \mathbb{Q}(\alpha,\sqrt{-n})$.
\par Assume $2 \mid d_M$. By the Dedekind-Kummer theorem, 
the prime $2$ is ramified over $M$. So $\langle 2 \rangle$ 
decomposes either as $\mathfrak{P}_1^2\mathfrak{P}_2,$ or as  $\mathfrak{P}^3$.   
\par We first consider the case where $\langle 2 \rangle=\mathfrak{P}_1^2\mathfrak{P}_2.$ 
Let $\mathfrak{P}$ be a prime ideal above $p$  not dividing $2n$. If $\mathfrak{P}$ is ramified over $L,$ then $e(\mathfrak{B}\mid
\mathfrak{P}) =3$ since $L/K$ is a Galois extension ($\mathfrak{B}$ is a prime ideal of $L$ above $\mathfrak{P}$). One deduces that $e(\mathfrak{B}\mid p) \geq3$. Since 
the ramification index $e(\cdot)$ is a multiplicative function we obtain a contradiction and deduce that 
all prime ideals of $K$ above $p\nmid 2n$ are unramified. Let $\mathfrak{P}$ be a prime ideal of $K$ above $p \mid n$. We assume 
$\mathfrak{P}$ ramifies over $L$ (and so $e(\mathfrak{B}\mid
\mathfrak{P}) =3$). Since $p \mid d_K$, we obtain by 
the Dedekind-Kummer theorem  that $p$ is ramified and so totally ramified over $L$. Let $\mathfrak{P}'$ be a prime ideal of $M$ such that $\mathfrak{B}$ is above $\mathfrak{P}'$. Using the fact that the ramification index $e(\cdot)$ is a multiplicative function, we get $e(\mathfrak{P}' \mid p) =3$. The different of the extension $M/\mathbb{Q}$ is divisible by 
$\mathfrak{P}'^2$ by Proposition 8 of \cite{L}.
Since the norm of the different equals the discriminant (see for example Proposition 14 of \cite{L}), we deduce that the exponent of $p$ in $n$ is  greater than $2,$ contradicting $n$ 
being a square-free integer. So all the prime ideals of $K$ above $p$ dividing $n$ 
are unramified. If $-n \equiv 1 \pmod*{4},$ then $d_K=-n$ and so $2 \nmid d_K$. If $-n \equiv 1 \pmod*{8}$, then we have  $\langle 2 \rangle=\mathfrak{P}_1'\mathfrak{P}_2'$. Assume that 
the $\mathfrak{P}_i'$ are unramified over $L$. Let $\mathfrak{B}'$ be  a prime ideal of $L$ above $\mathfrak{P}_i'$. Since $L/K$ is Galois and the ramification index $e(\cdot)$ is a multiplicative function, one gets a contradiction. So the $\mathfrak{P}_i'$ are ramified over $L$. Since $(2,3)=1$, we deduce by the definition of the conductor that $\nu_{\mathfrak{P}_i'}(\mathfrak{F})=1,$ and hence $\mathfrak{F}=\mathfrak{P}_1'\mathfrak{P}_2'.$ If $-n \equiv 5 \pmod*{8}$, then $\langle 2 \rangle$ is prime ideal over $K$. By the same argument as above, we 
deduce that $\langle 2 \rangle$ ramifies and so $\mathfrak{F}=\langle 2 \rangle$. Now, if $-n \equiv 3 \pmod*{4}$, then $d_K=-4n$ and so $2$ ramify according to Dedekind-Kummer theorem i.e. $\langle 2 \rangle= \mathfrak{P}^2$. If $\mathfrak{P}$ is ramified over $L$, then $2$ is totally ramified over $L,$ 
contradicting the assumption that $\langle 2 \rangle=\mathfrak{P}_1^2\mathfrak{P}_2$ over $M$. So 
$\mathfrak{P}$ is unramified and   $\mathfrak{F}=\langle 1 \rangle$.  
\par Assume now $\langle 2 \rangle=\mathfrak{P}^3$ over $M$. By the same argument as before, one deduces that all prime 
ideals $\mathfrak{P}'$ of $K$ above $p\neq 2$ are unramified. It follows that $\mathfrak{F}=\mathfrak{P}_1'\mathfrak{P}_2'$ if $-n \equiv 1 \pmod*{8}$ and $\mathfrak{F}=\langle 2 \rangle$ if $-n \equiv 5 \pmod*{8}$.  If  $-n \equiv 3 \pmod*{4}$, then $d_K=-4n$ and so 
$\langle 2 \rangle= \mathfrak{P}^2$ by the Kummer-Dedekind theorem. If $\mathfrak{P}$ is 
unramified over $L$, then we get the  contradiction 
using the multiplicativity of the ramification index. So $\mathfrak{P}$ is ramified over $L$ and $\mathfrak{F}=\mathfrak{P}$. 
Now, if $2 \nmid d_M$, then
we use the same argument as above to get $\mathfrak{F}=\langle 1 \rangle$. 
\par Let $H_{\mathfrak{F}}$ be the Hilbert class field of $K$ modulo $\mathfrak{F}$.  
By definition of the Hilbert class field,  we have $L\subseteq H_{\mathfrak{F}}.$
It is well known (see for example \cite{Mi}) that $$h_{\mathfrak{F}}= 
\frac{h_KN\mathfrak{F}}{(U:U_{\mathfrak{F}})}  \prod_{\mathfrak{P} \mid \mathfrak{F}}^{}  \left( 1- \dfrac{1}{N(\mathfrak{P})} \right),$$ 
where $h_{\mathfrak{F}}= \# Cl_{\mathfrak{F}}$, $h_K$ 
is the class number of $K$, $U= \mathcal{O}_K^*,~~U_{\mathfrak{F}}= U \cap K_{\mathfrak{F}},$ with $$K_{\mathfrak{F}}:= \{ x \in K^* ~\mid~~ \nu_{\mathfrak{P}}(x)=0 \quad \mbox{for }~~\mbox{all} \quad \mathfrak{P} \mid \mathfrak{F} \},$$ and $(U:U_{\mathfrak{F}})$ denotes the cardinality of the quotient group $U/U_{\mathfrak{F}}$. As $U=\{ \pm 1\}$ 
we have $U_{\mathfrak{F}}=\{ \pm 1\}$.  
 By the Artin reciprocity law $[H_{\mathfrak{F}}:K]$ equals 
$h_{\mathfrak{F}}.$ 

\noindent $\star$ \textbf{Case 1:} $2 \mid d_M, -n \equiv 1 \pmod*{8}$ and $h_K=3$. 
\par By Theorem \ref{thm1}, we have $\mathfrak{F}= \mathfrak{P}_1'\mathfrak{P}_2'$ and so $$[H_{\mathfrak{F}}:K] = 3 \cdot4\cdot (1-1/2) 
( 1-1/2)= 3= [L:K].$$ Since $L\subseteq H_{\mathfrak{F}},$ it follows that $L= H_{\mathfrak{F}}$.\\
$\star$ \textbf{Case 2:} $2 \mid d_M,~-n \equiv 5 \pmod*{8}$ and $h_K=1$. 
\par By Theorem \ref{thm1}, we have $\mathfrak{F}= \langle 2 \rangle$ and so $$[H_{\mathfrak{F}}:K] = 1 \cdot4\cdot (1-1/4)= 3= [L:K].$$ Since $L\subseteq H_{\mathfrak{F}},$ it follows that $L= H_{\mathfrak{F}}$.\\
$\star$ \textbf{Case 3:} $\langle 2 \rangle=\mathfrak{P}_1^2,\mathfrak{P}_2,~-n \equiv 3 \pmod*{4}$ and $h_K=3$.\par By Theorem \ref{thm1}, we have $\mathfrak{F}= \langle 1 \rangle$ and so $$[H_{\mathfrak{F}}:K] = h_K = 3 = [L:K].$$ Since $L\subseteq H_{\mathfrak{F}},$ it follows that $L= H_{\mathfrak{F}}$.\\
$\star$ \textbf{Case 4:}  $\langle 2 \rangle=\mathfrak{P}^3,~-n \equiv 3\pmod*{4}$ and $h_K=3$. 
\par By Theorem \ref{thm1}, we have $\mathfrak{F}= \mathfrak{P}$ and so $$[H_{\mathfrak{F}}:K] = 3 \cdot2\cdot (1-1/2)= 3= [L:K].$$ Since $L\subseteq H_{\mathfrak{F}},$ it follows that $L= H_{\mathfrak{F}}$.\\
$\star$ \textbf{Case 5:} $2 \nmid d_M$ and $h_K=3$. 
\par By Theorem \ref{thm1}, we have $\mathfrak{F}= \langle 1 \rangle$ and so $$[H_{\mathfrak{F}}:K] = h_K = 3 = [L:K].$$ Since $L\subseteq H_{\mathfrak{F}},$ it follows that $L= H_{\mathfrak{F}}$.\\

Thus we have dealt with all five cases completing the proof.
\end{proof}

%###########################################%
\section{Proof of Theorem \ref{thm3}}\label{sec:thm3}
%###########################################%
The proof is similar to that of Theorem 9.4 in the book of Cox \cite{Cox}.
\begin{proof}
We are going to show that the first assertion is equivalent to the second one and the first assertion is equivalent to the third one. Assume $p$ splits completely over $L$. Then it splits completely 
over $K$ and so $p= \mathfrak{P}_1\mathfrak{P}_2$, with $\mathfrak{P}_1$ and $\mathfrak{P}_2$ prime ideals of $\mathcal{O}_K$. By 
the Artin reciprocity law and Theorem \ref{thm1} it follows that $\mathfrak{P}_1=(a)$, where $a \in \mathcal{O}_K$ and $\nu_{\mathfrak{P}_i}(a-1)\geq \nu_{\mathfrak{P}_i}(\mathfrak{F})$.\\
\noindent $\star$ \textbf{Case 1:} $2 \mid d_M,~-n \equiv 1 \pmod*{8}$ and $h_K=3$. 
\par By Theorem \ref{thm1}, we have $\langle 2 \rangle \mid (a-1)$ and so $a-1=2t,$  with $t \in \mathcal{O}_K$. Since $-n \equiv 1 \pmod*{8}$, it follows that $\mathcal{O}_K = \mathbb{Z}\left[ \frac{1+\sqrt{-n}}{2}\right]$. 
Thus we can write
$t=X_1+X_2(1+\sqrt{-n})/2$ with $X_1, X_2 \in \mathbb{Z}$ and
so $a=(2X_1+X_2+1) + X_2\sqrt{-n}.$  Hence
$$p= N_{K/\mathbb{Q}}(\mathfrak{P}_1) = N_{K/\mathbb{Q}}(a)= X^2+nY^2,\,\text{with}\,X,Y \in \mathbb{Z}.$$
$\star$ \textbf{Case 2:} $2 \mid d_M,~-n \equiv 5 \pmod*{8}$ and $h_K=1$. 
\par We proceed as in Case 1 to get the result.\\
$\star$ \textbf{Case 3:} $\langle 2 \rangle=\mathfrak{P}_1^2,\mathfrak{P}_2,~-n \equiv 3 \pmod*{4}$ and $h_K=3$.
\par Since $-n \equiv 3 \pmod*{4}$, it follows that $\mathcal{O}_K = \mathbb{Z}\left[ \sqrt{-n}\right]$. Hence we get $a=X_1 + X_2\sqrt{-n},$ which implies 
 $$p= N_{K/\mathbb{Q}}(\mathfrak{P}_1) = N_{K/\mathbb{Q}}(a)= X_1^2+nX_2^2,\,\text{with}\,X_1,X_2
 \in \mathbb{Z}.$$ 
$\star$ \textbf{Case 4:} $\langle 2 \rangle=\mathfrak{P}^3,~-n \equiv 3\pmod*{4}$ and $h_K=3$.
\par We proceed as in 
Case 3 to get the desired result.\\
$\star$ \textbf{Case 5 (1) ($2 \nmid d_M$ and $-n \not \equiv 5 \pmod*{8}$ and $h_K=3$)}. 
We consider two subcases. If $-n \equiv 3\pmod*{4}$, then we have the result using the same argument as above. 
If $-n \equiv 1\pmod*{8}$, then  
$$\mathcal{O}_K = \mathbb{Z}\left[ \frac{1+\sqrt{-n}}{2}\right] \quad \mbox{and} \quad a= \left( \frac{2X_1+X_2}{2}\right)  +\left( \frac{X_2}{2}\right)  \sqrt{-n}.$$ Let us 
show that $X_2$ is an even integer in order to conclude that $(2X_1+X_2)/2$ and $X_2/2$ are integers,
which together with $p=N_{K/\mathbb{Q}}(a)$ will give us the desired result. 
If $X_2$ were an odd integer, then we would get $4p = X^2+nY^2$, with $X$ and $Y$ odd integers. Using the fact that $-n \equiv -1 \pmod*{8},$ we obtain $4p \equiv 0 \pmod*{8}$ and so $2 \mid p,$ which is a contradiction.
\par Now assume that $p=X^2+nY^2=(X+Y\sqrt{-n})(X-Y\sqrt{-n})$, with $X,Y \in \mathbb{Z}.$\\
 \noindent $\star$ \textbf{Case 1:} $2 \mid d_M,~-n \equiv 1 \pmod*{8}$ and $h_K=3$. 
\par By Theorem \ref{thm1}, we have $\langle 2 \rangle \mid (a-1)$ and $X+Y \equiv 1 \pmod*{2}$ since 
$p$ and $n$ are odd integers. Put $a = (1+\sqrt{-n})/2,$ 
we get $X+Y\sqrt{-n}=(X-Y)+2a$. By the Artin reciprocity law, we deduce that $(X+Y\sqrt{-n})$ splits completely over $L$. Likewise we show that $(X-Y\sqrt{-n})$ splits completely over $L$. \\
$\star$ \textbf{Case 2:} $2 \mid d_M$ and $-n \equiv 5 \pmod*{8}$ and $h_K=1$. 
\par We proceed as in Case 1 to get the desired result.\\
$\star$ \textbf{Case 3:} $\langle 2 \rangle=\mathfrak{P}_1^2,\mathfrak{P}_2$ and $-n \equiv 3 \pmod*{4}$ and $h_K=3$. 
\par Since $\mathcal{O}_K=\mathbb{Z}[\sqrt{-n}],$ the result follows immediately.\\
$\star$ \textbf{Case 4:} $\langle 2 \rangle=\mathfrak{P}^3,~-n \equiv 3\pmod*{4}$ and $h_K=3$.
\par Since $\mathcal{O}_K=\mathbb{Z}[\sqrt{-n}],$ the result follows immediately.\\
$\star$ \textbf{Case 5:} (1) $2 \nmid d_M,~-n \not \equiv 5 \pmod*{8}$ and $h_K=3$. 
\par We have $X+Y\equiv 1 \pmod*{2}$ 
since $p$ and $n$ are odd integers. We also have $\mathfrak{P}=\langle 2, 1+\sqrt{-n}\rangle$ since $-2=-4k+2\sqrt{-n} - 2 + (1-n-2\sqrt{-n}),$ with $-n+1= 4k$. 
Now we want to show that $X+Y\sqrt{-n}-1 \in \mathfrak{P}$. 
We can write $X-1=2a_1-Y,$ with $a_1 \in \mathbb{Z}$ 
as $X+Y$ is odd, 
and hence $$x+y\sqrt{-n}-1=2a_1-Y+Y\sqrt{-n}= 2a_1-Y(\sqrt{-n}-1) \in \mathfrak{P}.$$ So $\mathfrak{P} \mid (X+Y\sqrt{-n}-1)$ and we conclude that $(Y\pm X\sqrt{-n})$ splits completely over $L$. Hence $p$ splits completely over $L$, completing the proof of the first equivalence.

\par Suppose $p$ splits completely over $L$, then $p=X^2+nY^2$ for some $X,Y \in \mathbb{Z}$ which implies $\kronecker{-n}{p}=1.$  Moreover,  $p$ splits completely over $M$. If $f$ has no root 
in $\mathbb{Z}/p\mathbb{Z}$, then $p$ is a prime ideal of $M$ 
with inertia degree $3$ by Proposition 8.3  of \cite{JN}, contradicting the fact that $p$ splits completely over $M$. Assume that $\kronecker{-n}{p}=1$ and $f$ has a root in $\mathbb{Z}/p\mathbb{Z}$. The 
assumption that $\kronecker{-n}{p}=1$ 
ensures that $p=\mathfrak{P}_1\mathfrak{P}_2$, where $\mathfrak{P}_1, \mathfrak{P}_2$ are prime ideals of $\mathcal{O}_K$. Since $$\mathcal{O}_K/\mathfrak{P}_i \simeq \mathbb{Z}/p\mathbb{Z}, \quad \mbox{with} \quad i=1,2, $$ we deduce that 
$f$ has a root in $\mathcal{O}_K/\mathfrak{P}_i$. Using Proposition 8.3 of \cite{JN} 
and the fact that $L/K$ is a Galois extension, we infer that $\mathfrak{P}_i$ splits completely over 
$L,$ and so $p$ 
splits completely over $L$.

\par Now we treat the case where $2 \nmid d_M,-n \equiv 5 \pmod*{8}$ 
and $h_K=3$.  Suppose that $p$ splits completely over $K,$ and hence $p=\mathfrak{P}_1\mathfrak{P}_2$, where $\mathfrak{P}_1, \mathfrak{P}_2$ are 
prime ideals of $K$. By 
the Artin reciprocity law, $\mathfrak{P}_i=(a_i)$, where $a_i \in \mathcal{O}_K$. We can 
write $$a_i= \left( \frac{2X_1\pm X_2}{2}\right)  \pm \left( \frac{X_2}{2}\right)  \sqrt{-n},\quad X_1,X_2\in \mathbb Z.$$ Since $N_{K/ \mathbb{Q}}(\mathfrak{P}) =p,$ it follows that 
$$p= \left( \frac{2X_1\pm X_2}{2}\right)^2  + n\left( \frac{X_2}{2}\right)^2,$$  and $2X_1 \pm X_1 \pm X_2 \equiv 0 \pmod*{2}.$ So we have the result.   Now assume that  
$$p= \left( \frac{X}{2}\right)^2  + n\left( \frac{Y}{2}\right)^2= \left( \frac{X}{2} 
+\frac{Y}{2} \sqrt{-n}\right) \left( \frac{X}{2} -\frac{Y}{2} \sqrt{-n}\right),\quad X,Y\in \mathbb Z.$$ Put $a=(1+\sqrt{-n})/2.$ Thus $\sqrt{-n} = 2a-1$ and one obtains 
$$\frac{X}{2} +\frac{Y}{2} \sqrt{-n} = \frac{X-Y}{2} + aY\quad \text{and}\quad
\frac{X}{2} -\frac{Y}{2} \sqrt{-n} = \frac{X+Y}{2} - aY.$$ 
Since 
by assumption $X+Y$ is even, we get $$ \frac{X}{2} +\frac{Y}{2} \sqrt{-n} \in\mathcal{O}_K \quad \mbox{and} \quad  \frac{X}{2} -\frac{Y}{2}\sqrt{-n} \in\mathcal{O}_K.$$  
The corresponding ideals
$\left( \frac{X}{2} +\frac{Y}{2} \sqrt{-n} \right) $ and $\left( \frac{X}{2} -\frac{Y}{2} \sqrt{-n} \right) $ 
split completely over $L$ by the
Artin reciprocity law.
\end{proof}

%###########################################%
\section{Proof of Corollary \ref{cor1}}\label{sec:corollary}
%###########################################%
\begin{proof}
We consider the polynomial $f=x^3-x-1$. We have ${\rm disc}(f)=-23$ and $h_K = 3$, so, by Theorem \ref{thm1}, we deduce that the splitting field $L$ of $f$ is the Hilbert class field of $K$ modulo $\mathfrak{F}=\langle 1 \rangle.$ By Theorem \ref{thm3}, it follows that a prime $p\nmid 2\cdot 23$ splits completely over $L$ if and only if $p=X^2+23Y^2$, with $X,Y \in \mathbb{Z}$. So $2) \Leftrightarrow 3)$. Now assume $p$ splits completely over $L$. By Binet's formula, we have $P_n= \alpha^n + \beta^n + \gamma^n,$ where $\alpha, \beta, \gamma$ are the roots of $f$. Let $\mathfrak{B}$ be a prime ideal of $L$ above $p$. Since $p$ splits completely over $L$, it follows that the decomposition group of $\mathfrak{B}$ is trivial. So  $x \equiv x^{p} \pmod*{\mathfrak{B}}$ 
for every $x \in \mathcal{O}_L.$ 
Hence, modulo $\mathfrak{B},$
$$
	P_{p+1} \equiv  \alpha^2+  \beta^2+ \gamma^2 
	\equiv  (\alpha+\beta+\gamma)^2-2(\alpha \beta + \alpha \gamma + \beta \gamma)\equiv 0^2-2(-1)\equiv  2.$$
 Next suppose that $P_{p+1}\equiv2 \pmod*{p}$. By assumption $p\nmid 2\cdot 3\cdot 23$, so 
$p$ is unramified over $L,$ which implies that the decomposition group of $\mathfrak{B}$ is a cyclic group. 
Next, our aim is to show that $D(\mathfrak{B})$ is trivial. 
It is  clear that $D(\mathfrak{B})$ is a subgroup of $\text{Gal}(L/\mathbb{Q}),$ and so 
$\# D(\mathfrak{B}) \in \{1,2,3\}$ by Lagrange's theorem. If $\# D(\mathfrak{B})=2$, by symmetry, we can assume without loss of generality that the Frobenius element is given by $$\sigma(\alpha)= \beta,~~\sigma(\beta)= \alpha,~~\sigma(\gamma)= \gamma.$$ Hence, we have 
$$
	2 \equiv  \alpha^{p+1}+  \beta^{p+1}+ \gamma^{p+1}\\ 
	\equiv  \alpha \beta + \beta \alpha + \gamma^2 \\
	\equiv  2\alpha \beta +\gamma^2,
$$
Using the fact that $\alpha \beta + \alpha \gamma + \beta \gamma=-1,$ one 
gets modulo $\mathfrak{B}$ 
$$
4 \equiv -2(\alpha \gamma + \beta \gamma)+\gamma^2 \equiv   3\gamma^2.
$$
As $\gamma^2 \equiv 2-2\alpha \beta \pmod*{\mathfrak{B}},$  
it follows that $\alpha \beta \equiv -1/3 \pmod*{\mathfrak{B}}.$ Since $\alpha \beta \gamma =1$, we deduce that $\gamma \equiv -3 \pmod*{\mathfrak{B}}$. Hence we have 
$23/3 \equiv 0  \pmod*{\mathfrak{B}},$ contradicting $p\nmid 3\cdot 23.$ So  $\# D(\mathfrak{B}) \neq 2.$ Assume $\# D(\mathfrak{B}) = 3,$  we can assume without loss of generality that the Frobenius element is given by $$\sigma(\alpha)= \beta,~~\sigma(\beta)= \gamma,~~\sigma(\gamma)= \alpha.$$ Hence, we get 
modulo $\mathfrak{B}$ 
$$2 \equiv \alpha \beta + \alpha \gamma + \beta \gamma \equiv -1.$$ So $3\equiv 0 \pmod*{\mathfrak{B}},$ 
contradicting $p\neq 3$.
\end{proof}

%###########################################%
\section{Proof of Corollary \ref{cor2}}\label{sec:cor2}
%###########################################%
\begin{proof}
We consider the polynomial $f=x^3-2x^2+4x-4$. We have ${\rm disc}(f)=-4^2\cdot 11$, $h_K = 1$ and $\langle 2 \rangle=\mathfrak{P}^3$ in $M$ (see Proposition 10.5.2 \cite{A}), so, by Theorem \ref{thm1}, we deduce that the splitting field $L$ of $f$ is the Hilbert class field of $K$ modulo $\mathfrak{F}=\langle 2 \rangle.$ By Theorem \ref{thm3}, it follows that a prime $p\nmid 2\cdot 3\cdot 11 \cdot 13$ splits completely over $L$ if and only if $p=X^2+11Y^2$, with $X,Y \in \mathbb{Z}$. It is easy to see that $2)$ implies  $1)$.  By Binet's formula, we have $\sqrt{-176}\,B_n= (\gamma - \beta)\alpha^n + (\alpha - \gamma)\beta^n + (\beta - \alpha)\gamma^n,$ where $\alpha, \beta, \gamma$ are the roots of $f$. Let $\mathfrak{B}$ be a prime ideal of $L$ above $p$.  Now let's suppose that $B_{p}\equiv 0 \pmod*{p}$. The assumption on $p$ ensures that it is unramified over 
$L,$ and so the decomposition group of $\mathfrak{B}$ is a cyclic group. Now, we want to show that $D(\mathfrak{B})$ is trivial. It  is 
clear that $D(\mathfrak{B})$ is a subgroup of $\text{Gal}(L/\mathbb{Q})$ and so 
$\# D(\mathfrak{B}) \in \{1,2,3\}$ by Lagrange's theorem. If $\# D(\mathfrak{B})=2$, by symmetry, we can assume without loss of generality that the Frobenius element is given by $$\sigma(\alpha)= \beta,~~\sigma(\beta)= \alpha,~~\sigma(\gamma)= \gamma.$$ Hence, we have modulo $\mathfrak{B}$
$$
0 \equiv  (\gamma - \beta)\beta + (\alpha - \gamma)\alpha + (\beta - \alpha)\gamma \equiv 2\gamma (\beta - \alpha) + (\alpha-\beta)(\alpha+\beta).
$$
Since $\alpha \not \equiv \beta  \pmod*{\mathfrak{B}}$, it follows that $\alpha+\beta  \equiv 2\gamma \pmod*{\mathfrak{B}}$ and so $\gamma \equiv 2/3 \pmod*{\mathfrak{B}}$. Using the fact that $\alpha \beta \gamma =4$, we obtain $\alpha \beta \equiv 6 \pmod*{\mathfrak{B}}.$ However, $\gamma (\alpha + \beta) + \alpha \beta \equiv 6 \pmod*{\mathfrak{B}}$ and so $(2\cdot 13)/9 \equiv 0 \pmod*{\mathfrak{B}},$ 
contradicting our assumption that $p \nmid 2\cdot 3\cdot 13$. So  $\# D(\mathfrak{B}) \neq 2.$ 
Suppose that $\# D(\mathfrak{B}) = 3.$  Then we can assume without loss of generality that the Frobenius element is given by $$\sigma(\alpha)= \beta,~~\sigma(\beta)= \gamma,~~\sigma(\gamma)= \alpha.$$ Hence, we get 
modulo $\mathfrak{B}$
\begin{eqnarray*}
	0 &\equiv & (\gamma - \beta)\beta + (\alpha - \gamma)\gamma + (\beta - \alpha)\alpha \\ 
	&\equiv & \gamma \beta + \beta \alpha + \gamma \alpha - (\alpha^2+\beta^2 + \gamma^2) \\
	&\equiv & 3(\gamma \beta + \beta \alpha + \gamma \alpha) + (\alpha +\beta + \gamma)^2 \\
	&\equiv & 3\cdot 4+4,
\end{eqnarray*}
 contradicting our assumption that $p>2,$ 
 and so $\# D(\mathfrak{B})=1.$ 
 \end{proof}

%###########################################%
\section{Proof of Corollary \ref{cor5}}\label{sec:cor5}
%###########################################%
\begin{proof}
We consider the polynomial $f=x^3-x^2-1$. We have ${\rm disc}(f)=-31$ and $h_K = 3$, so, by Theorem \ref{thm1}, we deduce that the splitting field $L$ of $f$ is the Hilbert class field of $K$ modulo $\mathfrak{F}=\langle 1 \rangle.$ By Theorem \ref{thm3}, it follows that a prime $p\nmid 2\cdot 3 \cdot 29\cdot 31$ splits completely over $L$ if and only if $p=X^2+31Y^2$, with $X,Y \in \mathbb{Z}$. It 
is easy to see that $ 2)$ implies $1)$. Now assume that $C_p \equiv 0 \pmod*{p}.$ By Binet's formula, we have $\sqrt{-31}\,B_n= (\gamma - \beta)\alpha^n + (\alpha - \gamma)\beta^n + (\beta - \alpha)\gamma^n,$ where $\alpha, \beta, \gamma$ are the roots of $f$. Let $\mathfrak{B}$ be a prime ideal of $L$ above $p$.   
The assumption on $p$ ensures that it is unramified over 
$L,$ which implies that the decomposition group of $\mathfrak{B}$ is a cyclic group. Now, we want to show that $D(\mathfrak{B})$ is trivial. It  is clear 
that $D(\mathfrak{B})$ is a subgroup of $\text{Gal}(L/\mathbb{Q})$ and so 
 $\# D(\mathfrak{B}) \in \{1,2,3\}$ by Lagrange's theorem. If $\# D(\mathfrak{B})=2$, by symmetry, we can assume without loss of generality that the Frobenius element is given by $$\sigma(\alpha)= \beta,~~\sigma(\beta)= \alpha,~~\sigma(\gamma)= \gamma.$$ Hence, we have modulo $\mathfrak{B}$
$$
	0 \equiv  (\gamma - \beta)\beta + (\alpha - \gamma)\alpha + (\beta - \alpha)\gamma  
	\equiv  2\gamma (\beta - \alpha) + (\alpha-\beta)(\alpha+\beta).
$$
Since $\alpha\not \equiv \beta \pmod*{\mathfrak{B}}$, it now follows that $\alpha+\beta  \equiv 2\gamma \pmod*{\mathfrak{B}},$ which 
together with $1 = \alpha+\beta + \gamma$ yields $\gamma \equiv 1/3 \pmod*{\mathfrak{B}}.$ 
Since $\alpha \beta \gamma =1$, we obtain $\alpha \beta \equiv 3 \pmod*{\mathfrak{B}}.$ 
However, $\gamma (\alpha + \beta) \equiv 3 \pmod*{\mathfrak{B}},$ which implies $-29/3 \equiv 0 \pmod*{\mathfrak{B}},$ contradicting our assumption that $p \nmid 3\cdot 29$. So  $\# D(\mathfrak{B}) \neq 2.$ 
\par Assume $\# D(\mathfrak{B}) = 3.$ 
Without loss of generality we can assume 
that the Frobenius element is given by $$\sigma(\alpha)= \beta,~~\sigma(\beta)= \gamma,~~\sigma(\gamma)= \alpha.$$ Hence, we get  modulo $\mathfrak{B}$:
\begin{eqnarray*}
	0 &\equiv & (\gamma - \beta)\beta + (\alpha - \gamma)\gamma + (\beta - \alpha)\alpha \\ 
	&\equiv & \gamma \beta + \beta \alpha + \gamma \alpha - (\alpha^2+\beta^2 + \gamma^2) \\
	&\equiv & 3(\gamma \beta + \beta \alpha + \gamma \alpha) + (\alpha +\beta + \gamma)^2 \\
	&\equiv & 3\cdot 0+(-1)^2,
\end{eqnarray*}
which is impossible, and hence
$ \# D(\mathfrak{B})=1,$ 
yielding the equivalence
of (a) and (b).
\par The equivalence 
of (b) and (c) is a consequence of a variation
of the Ramanujan-Wilton congruence, see, e.g., 
Aygin and Williams \cite{AW} or Ciolan et al.\,\cite[Section 4.4]{clm}.
\end{proof}

\noindent \textbf{Acknowledgement}. 
This article has its origin in the second author being very intrigued
by an article of Evink and Helminck \cite{EH}.
The authors thank Francesco Campagna and Peter Stevenhagen for
a long and intense brainstorming session on the contents of this
paper. Further, they thank Seiken Saito, Zhi-Hong Sun and Jan Vonk for helpful e-mail correspondence. Sun pointed out his papers \cite{sun,sun3} and gave some examples of his results.
\par The bulk of the work on this paper was done during
a research stay of the second author at the Max Planck Institute for Mathematics (MPIM) in
the fall of 2022 and January 2023. 
He thanks this institution for the fruitful atmosphere of collaboration.

\bibliographystyle{plain}

\end{document}